\documentclass[12pt,a4paper]{article}
\usepackage[utf8]{inputenc}
\usepackage[a4paper, total={6in, 8in}]{geometry}
\usepackage{xcolor,amsmath}
\usepackage{amsthm}
\usepackage{amsfonts}
\usepackage{amssymb}
\usepackage{enumitem}
\usepackage{cite}
\usepackage{relsize}
\usepackage{blindtext, titlefoot}
\usepackage{hyperref}
\usepackage{sectsty}
\sectionfont{\fontsize{12}{15}\selectfont}

\newtheorem{lemma}{Lemma}
\newtheorem{theorem}{Theorem}

\newtheorem*{corollary*}{Corollary}

\newtheorem{proposition}{Proposition}
\author{Athanasios Sourmelidis\\
 {\small sourmelidis@math.tugraz.at}}
\date{}
\title{Discrete Moments of the Riemann Zeta-Function near $\sigma=1$}
\begin{document}
\maketitle
\begin{abstract}
\noindent
We employ mean value estimates of Weyl sums in order to obtain discrete second moments of the Riemann zeta-function with respect to polynomials near the vertical line $1+i\mathbb{R}$.
\end{abstract}
\keywords{\textbf{Keywords:}  Vinogradov's Mean Value Theorem,  Discrete Moments
\section{Introduction and Main Results}

Let $\zeta(s)$ denote the Riemann Zeta-function, where $s:=\sigma+it$ is a complex variable.
Our aim is to study discrete second moments of $\zeta(s)$ inside the vertical strip $\mathcal{D}:=\left\{s\in\mathbb{C}:1/2<\sigma\leq1\right\}$ with respect to polynomials and monomials of degree larger than 1. 
That is, we want to evaluate the following limits
\begin{align*}
\lim\limits_{N\to\infty}\dfrac{1}{N}\sum\limits_{n=1}^{N}\left|\zeta\left(s+i\left(a_1n+\dots+a_dn^d\right)\right)\right|^2\,\,\,\text{ and }\,\,\,\lim\limits_{N\to\infty}\dfrac{1}{N}\sum\limits_{n=1}^{N}\left|\zeta\left(s+ian^d\right)\right|^2,
\end{align*}
where $1/2<\sigma\leq 1$, $d\geq2$ is an integer and $a_1,\dots,a_d,a$ are positive real numbers.

The case of $d=1$ has been first resolved by Reich \cite{ReichE}, who proved that the limit exists for any $a>0$ and it coincides almost always with $\zeta(2\sigma)$. 
For the rest of the values $a$, the limit is equal to a factor, which depends on $a$, multiplied with $\zeta(2\sigma)$.  
One can already see that discrete moments are of different nature than the continuous moments, where we know that
\begin{align*}
\lim\limits_{T\to\infty}\dfrac{1}{T}\int\limits_{0}^{T}\left|\zeta\left(s+ia\tau\right)\right|^2\mathrm{d}\tau=\zeta(2\sigma)
\end{align*}
for any $1/2<\sigma\leq1$ and any $a>0$ (see for example \cite[Theorem 1.11]{ivicriemann}).
It is worth mentioning here the work of Good \cite{good1978diskrete}, who proved an asymptotic formula for the discrete fourth monents of  $\zeta(s)$ in $\mathcal{D}$ and his approach can easily be adapted for the discrete second moments of $\zeta(s)$, when someone wishes to have an error term as $N$ tends to infinity.

However, only the case of linear polynomials has been considered so far in the literature. 
The main reason is that, most of the times, we end up estimating finite exponential sums of the form
$
\sum_{n}\exp\left(2\pi i\left(a_1n+\dots+a_dn^d\right)\right)
$,
also known as Weyl sums, or
$
\sum_{n}\exp\left(2\pi i\left(an^d\right)\right)
$. 
 It is clear that in the case of $d=1$, that is, for linear polynomials, the latter sums coincide, represent a geometric series and can be estimated rather efficiently. 
As a  matter of fact, the treatment of the case $d=1$ has produced results also on the vertical line $1/2+i\mathbb{R}$, which is the most interesting case with respect to the zero-distribution of $\zeta(s)$.
 For example, Putnam \cite{zbMATH03087040} proved that the sequence of the positive consequtive zeros of the function $\zeta\left(1/2+it\right)$ does not contain any periodic subsequence, while van Frankenhuijsen \cite{zbMATH02244608} gave an explicit upper bound for the smallest positive integer $n$ such that $\zeta\left(1/2+ian\right)\neq0$ whenever $a>0$ is a fixed number.
 More recently, Li and Radziwi{\l\l} \cite{zbMATH06404852} proved by using the method of mollifiers that for any real numbers $a$ and $b$ with $a>0$,
 \begin{align*}
 \liminf\limits_{T\to\infty}\dfrac{1}{N}\sharp\left\{N<n\leq2N:\zeta\left(\dfrac{1}{2}+i(an+b)\right)\neq0\right\}\geq\dfrac{1}{3}.
 \end{align*}
 
 On the other hand, for $d\geq2$ the known estimates for Weyl sums rely on the rational approximations of the coefficients of the given polynomial $a_1x+\dots+a_dx^d$  and the corresponding length of the sum.  
 At this point lies the major difference with the case of the continuous moments of $\zeta(s)$, where exponential integrals
$\int\exp\left(2\pi i\left(a_1x+\dots+a_dx^d\right)\right)\mathrm{d}x$
have to be estimated and they are easier to handle (see for example \cite[Chapter 2]{ivicriemann}).
 At the expense of negligible subsets of $[0,+\infty)^d$ and $[0,+\infty)$, we will be able to overcome such difficulties, by making use of estimates for mean values of Weyl sums. 
Such estimates are based on the work of Bourgain, Demeter and Guth \cite{zbMATH06662221} regarding the main conjecture in Vinogradov's Mean Value Theorem as well as the work of Salberger and Wooley \cite{zbMATH05797868}. 

 \begin{theorem}\label{Discrete}
Let $d\geq2$ be an integer. Then there exists an effectively computable number $\mathbf{S}(d)\in(1/2,1)$, which tends to $1$ as $d$ tends to infinity, such that, for any $\mathbf{S}(d)<\sigma\leq1$ and almost every $(a_1,a_2,\dots,a_n)\in[0,+\infty)^d$,
\begin{align*}
\lim\limits_{N\to\infty}\dfrac{1}{N}\sum\limits_{n=1}^{N}\left|\zeta\left(s+i\left(a_1n+a_2n^2+\dots+a_dn^d\right)\right)\right|^2
=\zeta(2\sigma).
\end{align*}
\end{theorem}

\begin{theorem}\label{Discrete_mo}
Let $d\geq2$ be an integer.
Then there exists an effectively computable number $\mathbf{S}_{\mathrm{\mathbf{mo}}}(d)\in(1/2,1)$, which tends to $1$ as $d$ tends to infinity, such that, for any $\mathbf{S}_{\mathrm{\mathbf{mo}}}(d)<\sigma\leq1$ and almost every $a\in[0,+\infty)$,
\begin{align*}
\lim\limits_{N\to\infty}\dfrac{1}{N}\sum\limits_{n=1}^{N}\left|\zeta\left(s+ian^d\right)\right|^2
=\zeta(2\sigma).
\end{align*}
 \end{theorem}
 
Theorem \ref{Discrete} and Theorem \ref{Discrete_mo} come with two drawbacks. 
The first one is that we are able to prove them for almost all but not all vectors $(a_1,\dots,a_d)\in[0,+\infty)^d$, respectively numbers $a\in[0,+\infty)$.
  The second is with respect to the numbers $ \mathbf{S}(d)$ and $\mathbf{S}_{\mathrm{\mathbf{mo}}}(d)$, which, as we will also see in Section \ref{Lindel}, tend to $1$ from the left as $d$ tends to infinity. 
 It would be desirable to have $ \mathbf{S}(d)=\mathbf{S}_{\mathrm{\mathbf{mo}}}(d)=1/2$, for any $d\geq2$. 
 And indeed, under the Lindel\"of hypothesis, we can obtain such results.
  We will return to this discussion in the aforementioned section.
  
The discrete second moments of Theorem \ref{Discrete} and Theorem \ref{Discrete_mo} are always realized on the vertical line $1+i\mathbb{R}$.
  This is not unexpected for $d=1$ since we know that the sequence $(an)_{n\in\mathbb{N}}$ is {\it uniformly distributed modulo 1} for irrational numbers $a$ and the function 
  \begin{align*}
  \mathbb{R}\ni t\longmapsto\zeta(\sigma+it)
  \end{align*} 
  is {\it uniformly-almost-periodic} for $\sigma>1$ and $B^2${\it -almost-periodic} for $1/2<\sigma\leq1$.
  We refer to \cite{kuipers2012uniform} and \cite{zbMATH03110423} for the notions of uniformly distributed sequences and almost-periodic functions, respectively.
 It also comes by no surprise that Theorem \ref{Discrete} and Theorem \ref{Discrete_mo} are actually seen to hold whenever $d\geq2$ and $\sigma>1$. 
 However, this case is rather trivial as we will see in Section \ref{absolu}.
Our main purpose is to show that we can break through the $1+i\mathbb{R}$ "barrier`` when it is about shifts of $\zeta(s)$ with respect to polynomials of arbitrary degree and, more importantly, that the discrete second moments of $\zeta(s)$ are almost always independent of the polynomial in discussion.

The paper is organized as follows.
 In Section \ref{AFES} and Section \ref{Estimates of Exponential Sums} we lay the groundwork for the proofs of Theorem \ref{Discrete} and Theorem \ref{Discrete_mo} which will be given in Section \ref{proofs}. 
 In Section \ref{Lindel} we discuss some conditional improvements of our theorems and in Section \ref{absolu} we consider discrete second moments in the half-plane $\sigma>1$.
We conclude in Section \ref{genconc} with a few remarks regarding further possible generalizations of our results.  
  
\section{An Approximate Functional Equation}\label{AFES}
  The main result of this section is Proposition \ref{AFE} but before proceeding further we recall some common notations. 
  The {\it Vinogradov symbols} $\ll$ and $\gg$ have their usual meaning, and if the implied constant depends on some parameter $\epsilon$ (say), then we write $\ll_\epsilon$ to draw attention to this fact. 
  The same comments apply to the {\it Landau symbols} $o(\cdot)$ and $O(\cdot)$. 
  Moreover, all constants appearing in this section, implicit or not, are effectively computable.
  
 We define at first the function
  \begin{align}\label{A}
  (0,+\infty)\ni\mu\longmapsto\mathbf{A}(\mu):=\left\{
\renewcommand{\arraystretch}{1.5}
\begin{array}{ll}1-\mu^{-1},&\text{if}\,\,\,\mu\geq1,\\
\min\left\{\dfrac{1}{2\mu},1-\theta\mu^2\right\},&\text{if}\,\,\,0<\mu<1,
\end{array}
\right.
  \end{align}
  where $\theta:=4/(27\eta^2)$ and $\eta:=4.45$.
\begin{proposition}\label{AFE}
Let $\mu>0$ and $0<\delta<1-\mathbf{A}(\mu)$.
Then there exists a real number $\nu=\nu(\mu,\delta)>0$, such that
\begin{align*}
\zeta(s)
=\sum\limits_{0\leq n\leq t^\mu}\dfrac{1}{n^s}+O_{\mu,\delta}\left(t^{-\nu}\right),\,\,\,t\geq t_0>1,
\end{align*}
uniformly in $\mathbf{A}(\mu)+\delta\leq\sigma\leq1$.
\end{proposition}

The proof of the proposition is based on some well-known order estimates and approximate functional equations for $\zeta(s)$, as well as Perron's formula.
\begin{lemma}\label{sapprox}
If $\epsilon\in(0,1)$, then
\begin{align*}
\zeta(s)\ll_\epsilon|t|^\epsilon,\,\,\,\,\,|t|\geq t_0>0,
\end{align*}
uniformly in $1-\epsilon\leq\sigma\leq2$.
\end{lemma}

\begin{proof}
For a proof see \cite[Theorem 12.23]{apostol2013introduction}.
\end{proof}

The next lemma has its origins in the work of Vinogradov and Korobov regarding zero-free regions of $\zeta(s)$.
It has undergone through the decades many generalizations and improvements. We present here the latest version due to Ford  \cite[Theorem 1]{ford2002vinogradov}:

\begin{lemma}\label{apprx}
The following bound
\begin{align}\label{boundL}
\zeta(s)\ll |t|^{\eta(1-\sigma)^{3/2}}\left(\log |t|\right)^{2/3},\,\,\,\,\,|t|\geq t_0>1,
\end{align}
holds uniformly in $1/2\leq\sigma\leq1$.
\end{lemma}

\begin{proof}[Proof of Proposition \ref{AFE}]
The case of $\mu\geq1$ follows immediately from the classical approximate functional equation for $\zeta(s)$:
\begin{align*}
\zeta(s)=\sum\limits_{n\leq x}\dfrac{1}{n^s}+\dfrac{x^{1-s}}{s-1}+O_{\sigma_0}(x^{-\sigma}),\,\,\,t\geq t_0>0,
\end{align*}
where $0<\sigma_0\leq\sigma\leq2$ and $\pi x\geq t$ (see \cite[Theorem 1.8]{ivicriemann}).
We only need to set $x=t^\mu$.

If now $0<\mu<1$, then we see that
\begin{align*}
\mathbf{A}(\mu)=\left\{\begin{array}{lll}
\dfrac{1}{2\mu},&\text{if }\,\,\,\mu_0<\mu<1,\\
\\
1-\theta\mu^2,&\text{if}\,\,\,0<\mu\leq\mu_0,
\end{array}\right.
\end{align*}
where $\mu_0\in[1/2,3/4]$ is the unique real root of the polynomial $Q(x)=2\theta x^3-2 x+1$ in the interval $[0,1]$.

Let $\mu_0<\mu<1$. 
We consider the approximate functional equation for $\zeta(s)$ due to Hardy and Littlewood \cite[Theorem 1]{zbMATH02569960}:
\begin{align*}
\zeta(s)=&\,\sum\limits_{ n\leq x}\dfrac{1}{n^s}+\chi(s)\sum\limits_{n\leq y}\dfrac{1}{n^{1-s}}+O\left(x^{-\sigma}+t^{1/2-\sigma}y^{\sigma-1}\right),\,\,\,t\geq t_0>0,
\end{align*}
where $0\leq\sigma\leq1$, $x\geq1$ and $y\geq1$ are such that $2\pi xy=t$ and
\begin{align*}
\chi(s):=\pi^{s-1/2}\dfrac{\Gamma\left(\frac{1-s}{2}\right)}{\Gamma\left(\frac{s}{2}\right)},
\end{align*}
$\Gamma(s)$ being the Gamma function.
If we set $x=t^{\mu}$, then Stirling's formula (see for example \cite[(A.34)]{ivicriemann}) yields that
\begin{align*}
\zeta(s)-\sum\limits_{ n\leq t^\mu}\dfrac{1}{n^s}
&\ll t^{1/2-\sigma}\sum\limits_{n\leq t^{1-\mu}/(2\pi)}\dfrac{1}{n^{1-\sigma}}+t^{-\sigma\mu}+t^{1/2-\sigma+(\sigma-1)(1-\mu)}\\
&\ll t^{1/2-\sigma+ (1-\mu)\sigma}+t^{-\sigma\mu}+t^{-\sigma\mu+\mu-1/2}\\
&\ll t^{1/2-\sigma\mu}+t^{-1/2-\sigma\mu+\mu},\hspace{5.2cm}t\geq t_0>1,
\end{align*}
uniformly in $0\leq\sigma\leq1$.
In particular, the exponents of $t$ in the latter relation are negative for $$\dfrac{1}{2\mu}<\sigma\leq1.$$

Lastly, let $0<\mu\leq\mu_0$.
If we set $x:=m+1/2$, $m\in\mathbb{N}$, and $c:=1+1/\log(2 x)$, then the absolute convergence of $\zeta(s)$ in the half-plane $\sigma>1$ and a truncated version of Perron's formula (see \cite[Chapter III, Lemma 3.12]{titchmarsh1951theory}) yield that
\begin{align}\label{Perron2}
\dfrac{1}{2\pi i}\int\limits_{c-iT}^{c+iT}\zeta(s+z)\dfrac{x^z}{z}\mathrm{d}z
=\sum\limits_{n=1}^{m}\dfrac{1}{n^s}+O_{\sigma_0}\left(\dfrac{x\log x}{T}\right),\,\,\,T,x\geq t_0>1,
\end{align}
uniformly in $\sigma\geq\sigma_0>0$.

Let  $\sigma\in[1-\kappa+\delta,2]$ be arbitrary, where $\kappa=\kappa(\mu)\in(0,1/6)$ will be explicitly given in the end, and $0<\delta<\kappa$ is fixed. 
Let also $T=2 t$ and consider the rectangle $\mathcal{R}$ with vertices $1-3\kappa-\sigma\pm iT$, $c\pm iT$. 
By the calculus of residues we get
\begin{align}\label{resid}
\begin{split}
\dfrac{1}{2\pi i}\int_{\mathcal{R}}\zeta(s+z)\dfrac{x^z}{z}\mathrm{d}z
=\zeta(s)+\dfrac{x^{1-s}}{1-s}=\zeta(s)+O\left(\dfrac{x^{1-\sigma}}{t}\right),\,\,\,x,t\geq t_0>0.
\end{split}
\end{align} 
Now Lemma \ref{sapprox} implies that
\begin{align}\label{Lin2}
\left\{\int\limits_{1-3\kappa-\sigma-iT}^{c-iT}+\int\limits_{c+iT}^{1-3\kappa-\sigma+iT}\,\right\}\zeta(s+z)\dfrac{x^z}{z}\mathrm{d}z
\ll_\kappa \frac{x^{c}T^{3\kappa}}{T}\ll_\kappa \frac{xt^{3\kappa}}{t},\,\,\,x,t\geq t_0>0,
\end{align}
while Lemma \ref{apprx} yields that
\begin{align}\label{Lin3}
\begin{split}
\int\limits_{1-3\kappa-\sigma+iT}^{1-3\kappa-\sigma-iT}\zeta(s+z)\dfrac{x^z}{z}\mathrm{d}z
&\ll x^{1-3\kappa-\sigma}\int\limits_{-2t}^{2t}\dfrac{\left|\zeta\left(1-3\kappa+i(t+u)\right)\right|}{\left|1-3\kappa-\sigma+iu\right|}\mathrm{d}u\ll_\kappa\dfrac{t^{(3\kappa)^{3/2}\eta}\left(\log t\right)^2}{x^{2\kappa+\delta}},
\end{split}
\end{align}
for $x,t\geq t_0>1$.
From relations (\ref{Perron2})-(\ref{Lin3}) we deduce that
\begin{align*}
\zeta(s)=&\,\sum\limits_{n=1}^{m}\dfrac{1}{n^s}
+O_{\kappa}\left(\dfrac{x\log x+x^{1-\sigma}+xt^{3\kappa}}{t}+\dfrac{t^{(3\kappa)^{3/2}\eta}\left(\log t\right)^2}{x^{2\kappa+\delta}}\right),\,\,\,x,t\geq t_0>1,
\end{align*}
uniformly in $1-\kappa+\delta\leq\sigma\leq2$.
If we set $m=\lfloor t^\mu\rfloor$, then the last term is bounded from above by
\begin{align*}
C(\kappa,\mu,\delta)\left(t^{\mu-1}\log t
+t^{\mu(1-\sigma)-1}+
t^{\mu+3\kappa-1}+t^{\kappa\left(-2\mu+3^{3/2}\kappa^{1/2}\eta\right)-\mu\delta}\left(\log t\right)^2\right),
\end{align*}
where $C(\kappa,\mu,\delta)>0$ is a constant.
It is clear now that for $$
\kappa=\dfrac{4\mu^2}{27\eta^2}=\theta\mu^2$$
the proposition follows also for $0<\mu\leq\mu_0<3/4$.
\end{proof}

\section{Metric Results for Exponential Sums}\label{Estimates of Exponential Sums}

Bounds for exponential sums, especially in the case of Weyl sums, lie at the heart of analytic number theory. 
There is a broad literature regarding methods to estimate them as well as their numerous applications. We refer to \cite{ivicriemann, iwaniec2004analytic, zbMATH05013256} for an exposition of such results.

 In our case we focus on mean value estimates for Weyl sums, that is, on upper bounds for the quantity
 \begin{align*}
 J_{h,d}(N):=\int_{[0,1]^d}\left|\sum\limits_{n=1}^N\exp\left(2\pi i\left(a_1n+\dots+a_dn^d\right)\right)\right|^{2h}\mathrm{d}{a_1}\dots\mathrm{d}a_d,
\end{align*} 
where $h,d$ and $N$ are positive integers. Observe that $J_{h,d}(N)$ denotes the number of integral solutions of the system
\begin{align*}
\begin{array}{ccccccccccc}
X_1&+&\dots&+&X_h&=&X_{h+1}&+&\dots&+&X_{2h}\\
X_1^2&+&\dots&+&X^2_h&=&X^2_{h+1}&+&\dots&+&X^2_{2h}\\
&&\vdots&&&\vdots&&&\vdots&&\\
X_1^d&+&\dots&+&X_h^d&=&X_{h+1}^d&+&\dots&+&X_{2h}^d
\end{array}
\end{align*}
with $1\leq X_1,\dots,X_{2h}\leq N$.
 Recently, Bourgain, Demeter and Guth \cite{zbMATH06662221} proved the so-called main conjecture in Vinogradov's Mean Value Theorem:
\begin{theorem}[Bourgain, Demeter and Guth]\label{BDG}
For any integers $h\geq1$ and $d,N\geq2$,
\begin{align*}
J_{h,d}(N)\ll_{h,d,\epsilon}N^{h+\epsilon}+N^{2h-d(d+1)/2+\epsilon}.
\end{align*}
\end{theorem}

It should be noted here that the case of $d=2$ follows from elementary estimates for the divisor function, while the case of $d=3$ was first solved by Wooley \cite{zbMATH06567880}.

 We will use the latter theorem to obtain a useful metric result.
 Firstly, we introduce some notations which will be kept throughout the rest of the paper.
 We define the function $\mathrm{e}(x):=\exp(2\pi i x)$, $x\in\mathbb{R}$.
 Moreover, $P_{\underline{a}}(x):=a_1x+\dots+a_dx^d$ will denote a polynomial of given degree $d\in\mathbb{N}$, which correspond to a vector of real numbers $\underline{a}=(a_1,\dots,a_d)$.

\begin{lemma}\label{FLM}
 Let $d\geq2$ be an integer and $\epsilon,\mu>0$. 
There is a set $\mathcal{F}(d,\mu,\epsilon)\subseteq[0,+\infty)^d$ of full Lebesgue measure with elements satisfying the following property:

If $\underline{a}\in\mathcal{F}(d,\mu,\epsilon)$ is a vector of real numbers not exceeding an $M_{\underline{a}}\in\mathbb{N}$, then there exists $K_{\underline{a}}\in\mathbb{N}$ such that
\begin{align}\label{comp}
\left|\sum\limits_{n=1}^N\left(\dfrac{k}{\ell}\right)^{iP_{\underline{a}}(n)}\right|^{d(d+1)}
&\ll_{d,\epsilon}\left(\dfrac{\left\lfloor\frac{M_{\underline{a}}}{2\pi}\log\frac{k}{\ell}\right\rfloor+1}{\log\frac{k}{\ell}}\right)^{d}N^{d(d+1)/2+1+2\mu d+3\epsilon}
\end{align}
for every integer $N\geq K_{\underline{a}}$ and any integers $0\leq\ell<k\leq \left(2dM_{\underline{a}}N^d\right)^\mu$.
\end{lemma}
\begin{proof}
For any positive integers $M,N,k $, any integer $0\leq\ell<k$ and for $h=d(d+1)/2$, Theorem \ref{BDG} yields  that
\begin{align*}
\int_{[0,M]^d}\left|\sum\limits_{n=1}^N\left(\dfrac{k}{\ell}\right)^{iP_{\underline{a}}(n)}\right|^{d(d+1)}\mathrm{d}{\underline{a}}
&=\left(\dfrac{2\pi}{\log\frac{k}{\ell}}\right)^d\int_{\left[0,\frac{M}{2\pi}\log\frac{k}{\ell}\right]^d}\left|\sum\limits_{n=1}^N\mathrm{e}\left(P_{\underline{a}}(n)\right)\right|^{d(d+1)}\mathrm{d}{\underline{a}}\\
&\leq\left(\dfrac{2\pi\left(\left\lfloor \frac{M}{2\pi}\log\frac{k}{\ell}\right\rfloor+1\right)}{\log \frac{k}{\ell}}\right)^dJ_{h,d}(N)\\
&\leq C(d,\epsilon)\left(\dfrac{\left\lfloor \frac{M}{2\pi}\log\frac{k}{\ell}\right\rfloor+1}{\log \frac{k}{\ell}}\right)^dN^{d(d+1)/2+\epsilon},
\end{align*}
where $C(d,\epsilon)>0$ is a constant.
Therefore, the set
$\mathcal{E}(d,\mu,\epsilon,M,N,k,\ell)$ of those ${\underline{a}}\in[0,M]^d$ satisfying
\begin{align*}
\left|\sum\limits_{n=1}^N\left(\dfrac{k}{\ell}\right)^{iP_{\underline{a}}(n)}\right|^{d(d+1)}
\geq C(d,\epsilon)\left(\dfrac{\left\lfloor \frac{M}{2\pi}\log\frac{k}{\ell}\right\rfloor+1}{\log \frac{k}{\ell}}\right)^dN^{d(d+1)/2+1+2\mu d+3\epsilon},
\end{align*}
has Lebesgue measure $\mathrm{m}\left(\mathcal{E}(d,\mu,\epsilon,M,N,k,\ell)\right)\leq N^{-(1+2\mu d+\epsilon)}$.
 Hence, the set
\begin{align*}
\mathcal{G}(d,\mu,\epsilon,M,K):=\bigcup\limits_{N= K}^{\infty}\bigcup\limits_{1\leq\ell<k\leq \left(2dMN^d\right)^{\mu}}\mathcal{E}(d,\mu,\epsilon,M,N,k,\ell)
\end{align*}
has Lebesgue measure $\mathrm{m}(\mathcal{G}(d,\mu,\epsilon,M,K))\ll_{d,\mu,\epsilon,M} K^{-\epsilon}$ for every positive integers $M$ and $K$, and, thus, the set
\begin{align}\label{setting}
\mathcal{F}(d,\mu,\epsilon):=[0,+\infty)^d\setminus\left(\bigcup\limits_{M=1}^\infty\bigcap\limits_{K=1}^\infty\mathcal{G}(d,\mu,\epsilon,M,K)\right),
\end{align}
 is of full Lebesgue measure.
\end{proof}

 It is clear from the previous lemma that as soon as we have estimates for
\begin{align*}
M_{h,d}(N):=\int\limits_0^1\left|\sum\limits_{n=1}^N\mathrm{e}\left(an^d\right)\right|^{2h}\mathrm{d}a
\end{align*}
analogous to the one of Theorem \ref{BDG}, we can obtain similar metric results for monomials of degree larger than 1.
 Observe that
$M_{d,h}(N)$ denotes the number of $(2h)$-tuples $(X_1,\dots,X_{2h})$ of positive integers not exceeding $N$ for which
\begin{align*}
X_1^d+\dots+X_h^d=X_{h+1}^d+\dots+X_{2h}^d.
\end{align*}
In that direction, Salberger and Wooley \cite{zbMATH05797868} proved the following theorem:

\begin{theorem}[Salberger and Wooley]
Suppose that $d$ and $h$ are positive integers with $d\geq 2h-1\geq3$. Let also $T_h(N)$ denote the number of $(2h)$-tuples $(X_1,\dots,X_{2h})$ of positive integers not exceeding $N$ for which the $h$-tuple $(X_1,\dots,X_{h})$ is  a permutation of $(X_{h+1},\dots,X_{2h})$. Then
\begin{align*}
M_{h,d}(N)-T_h(N)\ll_{h,\epsilon}N^{h+\lambda(d,h)+\epsilon},
\end{align*}
where
\begin{align}\label{Woo}
\lambda(d,h):=\left\{
\begin{array}{lll}-2+2/\sqrt{3}+\kappa(d,h-1,2h-2),&\text{if}\,\,\,\,d\geq 2h-1,\\
-1+\kappa(d,h-1,2h-2),&\text{if}\,\,\,d\geq (2h-1)^2,\\
-\dfrac{1}{2},&\text{if}\,\,\,d\geq(2h)^{4h},
\end{array}
\right.
\end{align}
and
\begin{align*}
\kappa(d,k,m):=\sum\limits_{r=k+1}^m(r+1)/\sqrt[r]{d},
\end{align*}
for any positive integers $d,k$ and $m$ with $k<m$.
\end{theorem}
By definition $T_{h}(N)\sim h!N^h$. 
Therefore,
\begin{align*}
M_{h,d}(N)\ll_{h,\epsilon}N^{h+\max\lbrace0,\lambda(d,h)\rbrace+\epsilon},
\end{align*}
and we can prove a lemma similar to Lemma \ref{FLM}. We omit its proof.
\begin{lemma}\label{Wol}
Suppose that $d$ and $h$ are positive integers with $d\geq 2h-1\geq3$ and $\epsilon,\mu>0$. There is a set $\mathcal{F}_{\mathrm{\mathbf{mo}}}(d,\mu,\epsilon)\subseteq[0,+\infty)$ of full Lebesgue measure with elements satisfying the following property:

If $a\in\mathcal{F}_{\mathrm{\mathbf{mo}}}(d,\mu,\epsilon)$ is a real number bounded by an $M_a\in\mathbb{N}$, then there exists $K_a\in\mathbb{N}$ such that
\begin{align*}
\left|\sum\limits_{n=1}^N\left(\dfrac{k}{\ell}\right)^{ian^d}\right|^{2h}
&\ll_{h,\epsilon}\dfrac{\left\lfloor\frac{M_{a}}{2\pi}\log\frac{k}{\ell}\right\rfloor+1}{\log\frac{k}{\ell}}N^{h+\max\lbrace0,\lambda(d,h)\rbrace+1+2\mu d+3\epsilon}
\end{align*}
for every integer $N\geq K_{a}$ and any integers $0\leq\ell<k\leq \left(2dM_{a}N^d\right)^\mu$.
\end{lemma}

\section{Proofs of Theorem \ref{Discrete} and Theorem \ref{Discrete_mo} }\label{proofs}
\begin{proof}[Proof of Theorem \ref{Discrete}]
Let ${\underline{a}}\in\mathcal{F}(d,\mu,\epsilon)$ be such that $|a_i|\leq M_{\underline{a}}$ for some $M_{\underline{a}}\in\mathbb{N}$ and any $i=1,\dots,d$. 
 The numbers $\mu$ and $\epsilon$ will be suitably chosen in the end of the proof.
 
If $K_{\underline{a}}\in\mathbb{N}$ is as in Lemma \ref{FLM}, then let $N_{0}\gg\max\left\{ K_{\underline{a}},t\right\}$ be sufficiently large such that 
\begin{align*}
1\leq P_{\underline{a}}(N)+t\leq 2dM_{\underline{a}}N^d
\end{align*}
for all $N\geq N_{0}$.
 Notice here that all polynomials we consider are strictly increasing functions in $[0,+\infty)$.
 Using the approximate functional equation for $\zeta(s)$ from Proposition \ref{AFE} and applying the Cauchy-Schwarz inequality we obtain for every  $N\geq N_0$ that
\begin{align}\label{discrete}
\begin{split}
\sum\limits_{n=1}^{N}\left|\zeta\left(s+iP_{\underline{a}}(n)\right)\right|^2
=&\,\sum\limits_{n=N_{0}}^{N}\left|\sum\limits_{k\leq \left(t+P_{\underline{a}}(n)\right)^{\mu}}\dfrac{1}{k^{\sigma+{i\left(t+P_{\underline{a}}(n)\right)}}}+O_{\mu,\delta}\left(\left(t+P_{\underline{a}}(n)\right)^{-\nu}\right)\right|^2+\\
&+O_{t,\underline{a}}(1)\\
=&\,S_N+O_{t,\underline{a}}\left(1+T_N+\left(S_NT_N\right)^{1/2}\right),
\end{split}
\end{align}
valid for $\mathbf{A}(\mu)+\delta\leq\sigma\leq1$,
where $0<\delta<1-\mathbf{A}(\mu)$ and
$\nu=\nu(\mu,\delta)>0$ are as in Proposition \ref{AFE},
\begin{align}\label{T_N}
T_N:=\sum\limits_{n=N_{0}}^{N}\left|O_{\mu,\delta}\left(\left(t+P_{\underline{a}}(n)\right)^{-\nu}\right)\right|^2
\ll_{\mu,\delta,t,{\underline{a}}}\sum\limits_{n=N_{0}}^{N}n^{-2d\nu}
\ll_{\mu,\delta,t,{\underline{a}}}1+N^{1-2d\nu}
\end{align}
and
\begin{align*}
S_N&:=\sum\limits_{n=N_{0}}^{N}\sum\limits_{1\leq\ell,k\leq \left(t+P_{\underline{a}}(n)\right)^\mu}\dfrac{1}{(k\ell)^\sigma}\left(\dfrac{k}{\ell}\right)^{i\left(t+P_{\underline{a}}(n)\right)}.
\end{align*}
Splitting $S_N$ into sum of diagonal $(k=\ell)$ and non-diagonal terms $(k\neq\ell)$ yields that
\begin{align}\label{S_N}
\begin{split}
S_N
\hspace*{-0.5pt}&=\sum\limits_{n=N_{0}}^{N}\hspace*{-2.1pt}\left[\zeta(2\sigma)\hspace*{-2.1pt}+\hspace*{-2.1pt}O\left(\left(t+P_{\underline{a}}(n)\right)^{\mu(1-2\sigma)}\right)\right]\hspace*{-2.1pt}+\hspace*{-2.1pt}\sum\limits_{n=N_{0}}^{N}\sum\limits_{1\leq\ell\neq k\leq \left(t+P_{\underline{a}}(n)\right)^\mu}\dfrac{1}{(k\ell)^\sigma}\left(\dfrac{k}{\ell}\right)^{i\left(t+P_{\underline{a}}(n)\right)}\\
&=(N-N_{0})\zeta(2\sigma)+O_{\mu,\delta,t,\underline{a}}\left(\sum\limits_{n=N_{0}}^{N}n^{d\mu(1-2\sigma)}\right)+R_N\\
&=N\zeta(2\sigma)+O_{\mu,\delta,t,\underline{a}}\left(1+N^{1+d\mu(1-2\sigma)}\right)+R_N,
\end{split}
\end{align}
with
\begin{align*}
R_N:=\sum\limits_{1\leq\ell\neq k\leq \left(t+P_{\underline{a}}(N)\right)^\mu}\dfrac{1}{(k\ell)^\sigma}\left(\dfrac{k}{\ell}\right)^{it}\sum\limits_{n\in\mathcal{A}}\left(\dfrac{k}{\ell}\right)^{iP_{\underline{a}}(n)}
\end{align*}
and
\begin{align*}
\mathcal{A}:=\left\{N_{0}\leq n\leq N:P_{\underline{a}}(n)+t\geq\max\left\{k^{1/\mu},\ell^{1/\mu}\right\}\right\}:=\left\{{N}_1,{N}_1+1,\dots,N\right\}.
\end{align*}
Observe that
\begin{align}\label{Restglied}
\begin{split}
R_N\ll\sum\limits_{1\leq\ell<k\leq \left(2dM_{\underline{a}}N^d\right)^\mu}\dfrac{1}{(k\ell)^\sigma}\left[\left|\sum\limits_{n=1}^N\left(\dfrac{k}{\ell}\right)^{iP_{\underline{a}}(n)}\right|+\left|\sum\limits_{n=1}^{{N}_1-1}\left(\dfrac{k}{\ell}\right)^{iP_{\underline{a}}(n)}\right|\right].
\end{split}
\end{align}
By our choice of the vector ${\underline{a}}$  and Lemma \ref{FLM}, it follows that
\begin{align*}
\left|\sum\limits_{n=1}^N\left(\dfrac{k}{\ell}\right)^{iP_{\underline{a}}(n)}\right|
\ll_\epsilon\left(\dfrac{\left\lfloor \frac{M_{\underline{a}}}{2\pi}\log\frac{k}{\ell}\right\rfloor+1}{\log \frac{k}{\ell}}\right)^{1/(d+1)}N^{1/2+2\mu/(d+1)+1/(d(d+1))+\epsilon}
\end{align*}
for every $N\geq K_{\underline{a}}$ and any $1\leq\ell<k\leq \left(2d M_{\underline{a}} N^d\right)^{\mu }$.
Implementing the latter bound to \eqref{Restglied}, we obtain
\begin{align}
R_N\ll_\epsilon N^{1-2\mu d \mathbf{B}(d,\mu)+\epsilon}\widetilde{R}_N,
\end{align}
where
\begin{align}\label{B}
\mathbf{B}(d,\mu):=\dfrac{1}{2\mu d}\left(\dfrac{1}{2}-\dfrac{2\mu}{d+1}-\dfrac{1}{d(d+1)}\right)
\end{align}
and
\begin{align}\label{R_N}
\begin{split}
\widetilde{R}_N&:=\left(\frac{M_{\underline{a}}}{2\pi}\right)^{1/(d+1)}\sum\limits_{1\leq\ell< k\leq \left(2dM_{\underline{a}}N^d\right)^\mu}\dfrac{1}{(k\ell)^\sigma}\left(\dfrac{\left\lfloor\frac{M_{\underline{a}}}{2\pi}\log\frac{k}{\ell}\right\rfloor+1}{\frac{M_{\underline{a}}}{2\pi}\log\frac{k}{\ell}}\right)^{1/(d+1)}\\
  &\ll_{d,\underline{a}} \mathop{\sum\limits_{1\leq\ell< k\leq\left(2dM_{\underline{a}}N^d\right)^\mu}}_{\log(k/\ell)>2\pi/M_{\underline{a}}}\dfrac{1}{(k\ell)^\sigma}+\mathop{\sum\limits_{1\leq\ell< k\leq \left(2dM_{\underline{a}}N^d\right)^\mu}}_{\log(k/\ell)\leq2\pi/M_{\underline{a}}}\dfrac{1}{(k\ell)^\sigma}\left(\log\dfrac{k}{\ell}\right)^{-1}\\  
  &\ll_{d,\mu,\delta,\underline{a}} N^{(2-2\sigma)\mu d}(\log N)^2
  \end{split}
\end{align}
for every $N\geq N_{0}$. 
The square on the logarithm is justified when $\sigma=1$.

Gathering up the terms and estimates from \eqref{discrete}-\eqref{R_N}, we deduce that
\begin{align}\label{discrete1}
\begin{split}
\dfrac{1}{N}\sum\limits_{n=1}^{N}\left|\zeta\left(s+iP_{\underline{a}}(n)\right)\right|^2-\zeta(2\sigma)
\ll N^{-1/2}+N^{-d\nu}+N^{d\mu(1-2\sigma)}+N^{-2\mu d(\mathbf{B}(d,\mu)+\sigma-1)+\epsilon}
\end{split}
\end{align}
for every $N\geq N_{0}$, where the implicit constant depends on $d,\mu,\delta,t,{\underline{a}}$ and $\epsilon$.

Now we need to determine $\sigma\leq1$ and $\mu,\epsilon>0$ so that the right-hand side of \eqref{discrete1} is $o(1)$ as $N$ tend to infnity. The first three terms
$$N^{-1/2},\,\,\,N^{-d\nu},\,\,\,N^{d\mu\left(1-2\sigma\right)}$$
are $o(1)$ for any $1/2<\sigma\leq1$ and $\mu,\epsilon>0$. 
The last term
\begin{align}\label{lastt}N^{-2\mu d(\mathbf{B}(d,\mu)+\sigma-1)+\epsilon}
\end{align}
can be $o(1)$ for $\sigma\leq 1$ and $\mu>0$ only if $ \mathbf{B}(d,\mu)>0$, or equivalently from \eqref{B}, when
$$0<\mu<\dfrac{d^2+d-2}{4d}.$$

 It remains to choose $\mu_0=\mu_0(d)$ in such a way that we can obtain the wider strip possible containing the vertical line $1+i\mathbb{R}$, bound to the method we have used. 
 In view of \eqref{A}, \eqref{B} and the preceding discussion, the abscissa of the left boundary of the wider strip is given by
 \begin{align*}
  \mathbf{S}(d):=\inf\left\{\max\left\{\mathbf{A}(\mu),1-\mathbf{B}(d,\mu)\right\}:0<\mu<\dfrac{d^2+d-2}{4d}\right\}
 \end{align*}
 and the desired $\mu_0$ is the one for which the aforementioned well-defined infimum is obtained. 
 Observe then that for any $\delta>0$ such that
 $$\delta<1-\mathbf{S}(d)<1-A(\mu_0)$$
 and any $\mathbf{S}(d)+\delta\leq\sigma\leq1$, the term \eqref{lastt} is bounded above by 
 $$N^{-2\mu_0 d\delta+\epsilon}$$
 Now the theorem follows from \eqref{discrete1} and for $\epsilon=\mu_0 d\delta$.
\end{proof}

\begin{proof}[Proof of Theorem \ref{Discrete_mo}]
Now in order to obtain a theorem for the monomials $ax^d$, similar to Theorem \ref{Discrete}, we need to choose $2\leq h\leq(d+1)/2$ and $\mu>0$ in such way that the number
\begin{align}\label{B_mo}
\mathbf{B}_{\mathrm{\mathbf{mo}}}(d,h,\mu):=\dfrac{1}{2\mu d}\left(\dfrac{1}{2}-\dfrac{\max\lbrace0,\lambda(d,h)\rbrace+1+2\mu d}{2h}\right)
\end{align}
is positive.
The reasoning of how we obtain the quantity $\mathbf{B}_{\mathrm{\mathbf{mo}}}(d,h,\mu)$ is the same as  in \eqref{Restglied}-\eqref{B} of the previous proof, where this time we employ Lemma \ref{Wol} instead of Lemma \ref{FLM}.
From definition \eqref{Woo} of $\lambda(d,h)$  we know that for any $d\geq3$ and $h=2$
\begin{align*}
\dfrac{1}{2}-\dfrac{\max\lbrace0,\lambda(d,2)\rbrace+1}{4}>0.
\end{align*}
Hence, there is $2\leq h_{\mathrm{\mathbf{mo}}}\leq(d+1)/2$ such that 
\begin{align*}
e_{\mathrm{\mathbf{mo}}}:&=\dfrac{1}{2}-\dfrac{\max\lbrace0,\lambda(d,h_{\mathrm{\mathbf{mo}}})\rbrace+1}{2h_{\mathrm{\mathbf{mo}}}}\\&=\max\left\{\dfrac{1}{2}-\dfrac{\max\left\{0,\lambda(d,h)\right\}+1}{2h}:2\leq h\leq \dfrac{d+1}{2}\right\}\\&>0.
\end{align*}
If we set now 
 \begin{align}\label{S_mo}
  \mathbf{S}_{\mathrm{\mathbf{mo}}}(d):=\inf\left\{\max\left\{\mathbf{A}(\mu),1-\mathbf{B}_{\mathrm{\mathbf{mo}}}(d,h_{\mathrm{\mathbf{mo}}},\mu)\right\}:0<\mu<\dfrac{e_{\mathrm{\mathbf{mo}}}h_{\mathrm{\mathbf{mo}}}}{d}\right\},
 \end{align}
then the theorem follows in the same way as Theorem \ref{Discrete}.
 \end{proof}
 
Notice that in the statement of Theorem \ref{Discrete_mo} we also consider the case $d=2$ which can not be derived from Lemma \ref{Wol}. 
Instead, our starting point is a classical estimate due to Hua \cite{zbMATH02515008}:
 \begin{align*}
\int\limits_0^1\left|\sum\limits_{n=1}^N\mathrm{e}\left(an^2\right)\right|^{4}\mathrm{d}a\ll_\epsilon N^{2+\epsilon},
 \end{align*}
We can then argue as before to compute the corresponding quantity ${\bf S
}_{\mathrm{\mathbf{mo}}}(2)$.

\section{The Lindel\"of Hypothesis}\label{Lindel}
 As we have already seen in the previous sections, to control the length of the Dirichlet polynomials appearing in the approximate functional equation of Lemma \ref{AFE} and, in accordance, in the proofs of Theorem \ref{Discrete} and Theorem \ref{Discrete_mo}, we had to narrow the strip inside $\left\{s\in\mathbb{C}:1/2<\sigma\leq1\right\}$ for which we can prove discrete second moments.
Indeed, one could argue that there is some $\sigma_1$ with $1/2<\sigma_1<1$ such that $\mathbf{S}(d_k)\leq\sigma_1$ for infinitely many integers $d_k\geq2$, $k\in\mathbb{N}$. 
But that would require at first
\begin{align*}
\max\left\{\mathbf{A}(\mu_k),1-\mathbf{B}(d_k,\mu_k)\right\}\leq\sigma_1
\end{align*}
for all $k\in\mathbb{N}$ and some $0<\mu_k<(d_k^2+d_k-2)/(4d_k)$, or, equivalently,
\begin{align*}
\mathbf{A}(\mu_k)\leq\sigma_1\,\,\,\text{ and }\,\,\,\mu_k\leq\dfrac{d_k^2+d_k-2}{4d_k^2\left((1-\sigma_1)(d_k+1)+1\right)}
\end{align*}
for all $k\in\mathbb{N}$.
Then we would have $\lim\limits_{k\to\infty}\mu_k=0$ and, by the definition \eqref{A} of $\mathbf{A}(\mu)$, that $1=\lim\limits_{k\to\infty}\mathbf{A}(\mu_k)\leq\sigma_1$, which contradicts our assumption about $\sigma_1$.
Hence, $\mathbf{S}(d)$ tends to 1 from the left as $d$ tends to infinity. 
The same holds true for $\mathbf{S}_{\mathrm{\mathbf{mo}}}(d)$.

 The definitions \eqref{A} and \eqref{S_N} already give us a hint of how we could obtain that $ \mathbf{S}(d)=\mathbf{S}_{\mathrm{\mathbf{mo}}}(d)=1/2$, for any $d\geq2$.
It would suffice if we were able to take $\mathbf{A}(\mu)=1/2$ for any $0<\mu<1$, since $\lim\limits_{\mu\to0^+}\mathbf{B}(d,\mu)=\lim\limits_{\mu\to0^+}\mathbf{B}_{\mathrm{\mathbf{mo}}}(d,h_{\mathrm{\mathbf{mo}}},\mu)=+\infty$. 
This is where the {\it Lindel\" of hypothesis} for $\zeta(s)$  can be used:
\begin{align*}
\zeta\left(\dfrac{1}{2}+it\right)\ll_\epsilon|t|^\epsilon,\,\,\,|t|\geq t_0>0.
\end{align*}
In our case a classical result for $\zeta(s)$ (see \cite[Theorem 13.3]{titchmarsh1951theory}) was the inspiration for our Lemma \ref{AFE}, where the number $\mathbf{A}(\mu)$ first appears:

{\it The Lindel\" of hypothesis is true if and only if for any integer $k\geq1$, any $\sigma>1/2$ and any $0<\mu<1$, there is a positive number $\nu=\nu(k,\sigma,\mu)$ such that
\begin{align*}
\zeta^k(s)=\sum\limits_{n\leq t^\mu}\dfrac{d_k(n)}{n^s}+O(t^{-\nu}),\,\,\,t\geq t_0>0,
\end{align*}
where \begin{align*}
d_k(n)=\mathop{\sum\limits_{m_1,\dots,m_k\in\mathbb{N}}}_{m_1\cdots m_k=n}1,\,\,\,n\in\mathbb{N}.
\end{align*}}
It is then straightforward to obtain the following conditional results:

\begin{lemma}If the Lindel\" of hypothesis is true, the approximate functional equation of Lemma \ref{AFE} hold true with $\mathbf{A}(\mu)=1/2$ for any $0<\mu<1$.
\end{lemma}

\begin{theorem}
If the Lindel\" of hypothesis is true, Theorem \ref{Discrete} and Theorem \ref{Discrete_mo} hold true with $\mathbf{S}(d)=\mathbf{S}_{\mathrm{\mathbf{mo}}}(d)=1/2$, for any integer $d\geq2$.
\end{theorem}

\section{On the plane of absolute convergence}\label{absolu}
Since  the Dirichlet series representation of $\zeta(s)$ is absolutely convergent in the half-plane $\sigma>1$, computing its discrete second moments with respect to $P_{\underline{a}}(x)$ and $ax^d$ is much simpler if one is willing to omit a negligible set $[0,+\infty)^d\setminus\mathcal{L}(d)$, where
\begin{align*}
\mathcal{L}(d):=\bigcap\limits_{\underline{r}\in\mathbb{Q}^d}\mathop{\bigcap\limits_{k,\ell=0}^{\infty}}_{k\neq \ell}\left\{\underline{a}\in[0,+\infty)^d:a_i\log\dfrac{k}{\ell}\neq2\pi r_i,\,1\leq i\leq d\right\}.
\end{align*}
In the succeeding proof we will employ a useful property of polynomials with at least one irrational coefficient, namely that the sequence of their values on the integers is uniformly distributed modulo 1:
\begin{lemma}\label{help}
Let $P_{\underline{\alpha}}(x)=\alpha_1x+\dots+\alpha_dx^d$ be a polynomial of degree $d\geq1$ with real coefficients. 
If $\alpha_j$ is irrational for some $j=1,\dots,d$, then 
\begin{align*}
\lim\limits_{N\to\infty}\dfrac{1}{N}\sum\limits_{n=1}^N\mathrm{e}\left(P_{\underline{a}}(n)\right)=0.
\end{align*}
\end{lemma}
\begin{proof}
For a proof see \cite[Chapter 1, Theorem 3.2]{kuipers2012uniform} and \cite[Chapter 1, Theorem 2.1]{kuipers2012uniform}.
\end{proof}

Lastly, we will also partially treat the case when $\underline{a}\notin\mathcal{L}(d)$. 
For this purpose, we need to define here for any positive integers $k_0$ and $\ell_0$ the set
\begin{align*}
U(k_0,\ell_0):=\left\{\left(k,\ell\right)\in\mathbb{N}^2:\dfrac{k}{\ell}=\left(\dfrac{k_0}{\ell_0}\right)^{u_{(k,\ell)}}\,\,\,\text{ for some }\,\,\,u_{(k,\ell)}\in\mathbb{Z}\right\}.
\end{align*}
We are now ready to prove 
\begin{theorem}\label{s>1}
Let $d\geq1$ be an integer and $\sigma>1$.
If $\underline{a}\in\mathcal{L}(d)$, then
\begin{align}\label{indep}
\lim\limits_{N\to\infty}\dfrac{1}{N}\sum\limits_{n=1}^N\left|\zeta\left(s+iP_{\underline{a}}(n)\right)\right|^2=\zeta(2\sigma).
\end{align}
Let also $(m_1,\dots,m_d)\in\mathbb{Z}^d\setminus\left\{\underline{0}\right\}$ and $k_0\neq\ell_0$ be coprime positive integers such that $(k_0/\ell_0)^{1/q}$ is irrational for every $q\in\mathbb{N}\setminus\lbrace1\rbrace$.
If
\begin{align}\label{reson}
\underline{a}:=\left(2\pi m_1\left(\log\frac{k_0}{\ell_0}\right)^{-1},\dots,2\pi m_d\left(\log\frac{k_0}{\ell_0}\right)^{-1}\right),
\end{align}
then
\begin{align}\label{dep}
\lim\limits_{N\to\infty}\dfrac{1}{N}\sum\limits_{n=1}^N\left|\zeta\left(s+iP_{\underline{a}}(n)\right)\right|^2=\zeta(2\sigma)+2\mathop{\sum\limits_{(k,\ell)\in U(k_0,\ell_0)}}_{k<\ell}\dfrac{\cos\left(t\log\frac{k}{\ell}\right)}{(k\ell)^\sigma}.
\end{align}
\end{theorem}

\begin{proof}
Let $\underline{a}\in\mathcal{L}(d)$. 
Then
\begin{align*}
\lim\limits_{N\to\infty}\dfrac{1}{N}\sum\limits_{n=1}^N|\zeta\left(s+iP_{\underline{a}}(n)\right)|^2
&=\lim\limits_{N\to\infty}\dfrac{1}{N}\sum\limits_{n=1}^N\sum\limits_{k,\ell=1}^{\infty}\dfrac{1}{(k\ell)^\sigma}\left(\dfrac{k}{\ell}\right)^{i(t+P_{\underline{a}}(n))}\\
&=\sum\limits_{k,\ell=1}^{\infty}\dfrac{1}{(k\ell)^\sigma}\left(\dfrac{k}{\ell}\right)^{it}\lim\limits_{N\to\infty}\dfrac{1}{N}\sum\limits_{n=1}^N\left(\dfrac{k}{\ell}\right)^{iP_{\underline{a}}(n)}\\
&=:A,
\end{align*}
where interchanging summation and the limit operator is valid by the absolute convergence of the double series $\sum_{k,\ell}\left(k\ell\right)^{-\sigma}$.
 If we split now the latter double sum in sums of diagonal and non-diagonal terms, we get 
 \begin{align}\label{A1}
A=\zeta(2\sigma)+\mathop{\sum\limits_{k,\ell=1}^{\infty}}_{k\neq\ell}\dfrac{1}{(k\ell)^\sigma}\left(\dfrac{k}{\ell}\right)^{it}\lim\limits_{N\to\infty}\dfrac{1}{N}\sum\limits_{n=1}^N\mathrm{e}\left(\dfrac{P_{\underline{a}}(n)}{2\pi}\log\dfrac{k}{\ell}\right).
\end{align}
Observe that for any pair of integers $(k,\ell)$ with $k\neq\ell$, $\log\left(k/\ell)\right)P_{\underline{a}}(x)/2\pi$ is a polynomial with at least one irrational coefficient as can be seen from our choice of the vector $\underline{a}$. 
Therefore, Lemma \ref{help} yields that
\begin{align*}
\lim\limits_{N\to\infty}\dfrac{1}{N}\sum\limits_{n=1}^N\mathrm{e}\left(\dfrac{P_{\underline{a}}(n)}{2\pi}\log\dfrac{k}{\ell}\right)=0
\end{align*}
for any pair $(k,\ell)$ with $k\neq\ell$, and relation \eqref{indep} follows. 

Let now $\underline{a}$ be as in \eqref{reson}.
The proof of \eqref{dep} follows in the same lines as above until we reach relation \eqref{A1}.
In this case we split the sum of the non-diagonal terms into a sum $\sum_1$ over pairs $(k,\ell)\in U(k_0,\ell_0)$ with $k\neq\ell$ and a sum $\sum_2$ over pairs $(k,\ell)\in\mathbb{N}^2\setminus U(k_0,\ell_0)$:
\begin{align}\label{A_1}
\begin{split}
A=&\,\zeta(2\sigma)+\left\{\sum\nolimits_{1}+\sum\nolimits_2\right\}\dfrac{1}{(k\ell)^\sigma}\left(\dfrac{k}{\ell}\right)^{it}\lim\limits_{N\to\infty}\dfrac{1}{N}\sum\limits_{n=1}^N\mathrm{e}\left(\dfrac{P_{\underline{a}}(n)}{2\pi}\log\dfrac{k}{\ell}\right).
\end{split}
\end{align}
Observe that by our definition of $\underline{a}$ we have for every $(k,\ell)\in U(k_0,\ell_0)$ and any $n\in\mathbb{N}$ that
\begin{align}\label{1}
\mathrm{e}\left(\dfrac{P_{\underline{a}}(n)}{2\pi}\log\dfrac{k}{\ell}\right)=\mathrm{e}\left(u(k,\ell)\left(m_1n+\dots m_d n^d\right)\right)=1.
\end{align}
On the other hand, if $(k,\ell)\in\mathbb{N}^2\setminus U(k_0,\ell_0)$, then the polynomial $P_{\underline{a}}(x)\log(k/\ell)/2\pi$ has at least one irrational coefficient.
Assuming the contrary, we would have that for any $m_i\neq0$, there are coprime integers $p=p(k,\ell,m_i)$ and $q=q(k,\ell,m_i)$ such that
\begin{align*}
m_i\left(\log\dfrac{k_0}{\ell_0}\right)^{-1}\log\dfrac{k}{\ell}=\dfrac{p}{q}\,\,\,\text{ or }\,\,\,
\dfrac{k}{\ell}=\left(\dfrac{k_0}{\ell_0}\right)^{p/(qm_i)}
\end{align*}
which is impossible since the number $(k_0/\ell_0)^{p/(qm_i)}$ is irrational. 
Indeed, if $qm_i\mid p$ then $(k,\ell)\in U(k_0,\ell_0)$ and that contradicts our initial assumption on $(k,\ell)$.
Thus, $qm_i\nmid p$ and our claim follows from our assumptions that $(k_0,\ell_0)=1$, $k_0\neq\ell_0$ and that $(k_0/\ell_0)^{1/qm_i}$ is irrational.
Therefore, Lemma \ref{help} yields that
\begin{align}\label{Weyl1}
\lim\limits_{N\to\infty}\dfrac{1}{N}\sum\limits_{n=1}^N\mathrm{e}\left(\dfrac{P_{\underline{a}}(n)}{2\pi}\log\dfrac{k}{\ell}\right)=0
\end{align}
for any $(k,\ell)\in\mathbb{N}^2\setminus U(k_0,\ell_0)$.
Now relation \eqref{dep} follows from \eqref{A_1}-\eqref{Weyl1}.
\end{proof}
It should be mentioned here that Good \cite{good1978diskrete} gives a closed expression for the sum
\begin{align*}
\mathop{\sum\limits_{(k,\ell)\in U(k_0,\ell_0)}}_{k<\ell}\dfrac{\cos\left(t\log\frac{k}{\ell}\right)}{(k\ell)^\sigma}.
\end{align*}
It is too technical to be written down here, but it can be seen that it is equal to $\zeta(2\sigma)$ times a factor which depends on $t$ and the prime factorization of $k_0$ and $\ell_0$.
We could not treat the remaining, countably many cases of $\underline{a}\notin\mathcal{L}(d)$ for $d\geq2$, because we end up estimating Weyl sums $
\sum_{n}\mathrm{e}\left(a_1n+\dots+a_dn^d\right)
$ where the polynomials have only rational coefficients and at least one of them is not an integer.
Actually, if $a_2,\dots,a_d\in\mathbb{Z}$ and $a_1\in\mathbb{Q}\setminus\mathbb{Z}$, then
\begin{align*}
\sum_{n=1}^N\mathrm{e}\left(a_1n+\dots+a_dn^d\right)=\sum_{n=1}^N\mathrm{e}\left(a_1n\right)=O(1),\,\,\,N\geq 1,
\end{align*}
and we would obtain relation \eqref{indep}, as Reich \cite{ReichE} and Good \cite{good1978diskrete} did for the case $d=1$.
If, however, $a_i\in\mathbb{Q}\setminus\mathbb{Z}$ for some $i\in\lbrace2,\dots,d\rbrace$, then the known estimates for the corresponding Weyl sums are not sufficiently good to show that at least 
\begin{align*}
\sum_{n=1}^N\mathrm{e}\left(a_1n+\dots+a_dn^d\right)=o(N),\,\,\,N\geq 1.
\end{align*}
The analogous theorem for monomials can be obtained similarly:
\begin{theorem}
Let $d\geq1$ be an integer and $\sigma>1$.
If $a\in\mathcal{L}(1)$, then
\begin{align*}
\lim\limits_{N\to\infty}\dfrac{1}{N}\sum\limits_{n=1}^N\left|\zeta\left(s+ian^d\right)\right|^2=\zeta(2\sigma).
\end{align*}
Let also $m\in\mathbb{Z}\setminus\left\{{0}\right\}$ and $k_0\neq\ell_0$ be coprime positive integers such that $(k_0/\ell_0)^{1/q}$ is irrational for every $q\in\mathbb{N}\setminus\lbrace1\rbrace$.
If 
\begin{align*}
{a}:=2\pi m\left(\log\frac{k_0}{\ell_0}\right)^{-1},
\end{align*}
then
\begin{align*}
\lim\limits_{N\to\infty}\dfrac{1}{N}\sum\limits_{n=1}^N\left|\zeta\left(s+ian^d\right)\right|^2=\zeta(2\sigma)+2\mathop{\sum\limits_{(k,\ell)\in U(k_0,\ell_0)}}_{k<\ell}\dfrac{\cos\left(t\log\frac{k}{\ell}\right)}{(k\ell)^\sigma}.
\end{align*}
\end{theorem}

\section{Generalizations and Concluding Remarks}\label{genconc}
The main ingredients of this paper are an order result for $\zeta(s)$ (Lemma \ref{boundL}), bounds for mean values of Weyl sums (Theorem \ref{BDG} and Theorem \ref{Wol}) and basic properties of uniformly distributed sequences.

It is easy to see that all the results proven above will hold for any Dircihlet $L$-function with the same abscissas $\mathbf{S}(d)$ and $\mathbf{S}_{\mathrm{\mathbf{mo}}}(d)$.
The reason is that Lemma \ref{boundL} is in fact true for the Hurwitz zeta-function 
\begin{align*}
\zeta(s;\alpha):=\sum\limits_{n=0}^{\infty}\dfrac{1}{(n+\alpha)^s},\,\,\,\sigma>1,\alpha\in(0,1],
\end{align*}
and any Dirichlet $L$-function can be expressed as a finite linear combination of Hurwitz zeta-functions.
Therefore, we can use the same order result for such functions and the proofs remain the same.
As a matter of fact, the same can be said for the Hurwitz zeta-function of a given parameter $\alpha\in(0,1]$.
We will only need to alter slightly Proposition \ref{AFE}, Lemma \ref{FLM} and Lemma \ref{Wol}, where instead of $k/\ell$ the fractions $(k+\alpha)/(\ell+\alpha)$ appear.
Moreover, the results will also hold for any integral power $m$ of all the aforementioned functions.
In this case we will have to modify slightly $\mathbf{S}(d)$ and $\mathbf{S}_{\mathrm{\mathbf{mo}}}(d)$ by redefining the quantity $\mathbf{A}(\mu)$ in Proposition \ref{AFE} since in this case it will depend on $m$ as well. 
Again Lemma \ref{sapprox}, Lemma \ref{apprx} and the Lindel\"of hypothesis will yield the required results.

It is clear that any improvements on the number $\eta$ which appears in Lemma \ref{apprx}, that is, any improvements regarding the order of growth of $\zeta(s)$ inside the strip $1/2\leq\sigma\leq1$, would instantly improve the range of the vertical strips on which we can obtain discrete moments with respect to polynomials (unconditionally).
However, this is one of the most difficult parts on analytic number theory, it is strongly connected to zero-free regions of $\zeta(s)$ and, except of some distinguishable results due to Richert \cite{zbMATH03257526}, Kulas \cite{zbMATH01333807} and Ford \cite{ford2002vinogradov}, little has been established so far.

\bibliography{discrete_3}{}
\bibliographystyle{siam-new}
\end{document}